\documentclass[11pt,reqno]{amsart}
\usepackage{amsmath, amssymb, amsthm}
\usepackage{url}
\usepackage[breaklinks]{hyperref}

\renewcommand{\(}{\left\(}
\renewcommand{\)}{\right\)}
\renewcommand{\[}{\left\[}
\renewcommand{\]}{\right\]}
\renewcommand{\i}{\infty}
\numberwithin{equation}{section}
 \theoremstyle{plain}
\newtheorem{theorem}{Theorem}[section]
\newtheorem{lemma}[theorem]{Lemma}
\newtheorem{remark}[]{Remark}
\newtheorem{definition}[]{Definition}

\newtheorem{corollary}[theorem]{Corollary}
\newtheorem{proposition}[theorem]{Proposition}

   \makeatletter
\def\proof{\@ifnextchar[{\@oproof}{\@nproof}}
\def\@oproof[#1][#2]{\trivlist\item[\hskip\labelsep\textit{#2 Proof of\
#1.}~]\ignorespaces}
\def\@nproof{\trivlist\item[\hskip\labelsep\textit{Proof.}~]\ignorespaces}

\makeatother

\begin{document}
\title[Generalization of five $q$-series identities of Ramanujan]{Generalization of five $q$-series identities of Ramanujan and unexplored weighted partition identities}

\thanks{$2020$ \textit{Mathematics Subject Classification.} Primary 11P81, 11P84; Secondary 05A17.}

\author{Subhash Chand Bhoria}
\address{Subhash Chand Bhoria, Pt. Chiranji Lal Sharma Government PG College, Karnal, Urban Estate, Sector-14, Haryana 132001, India.}
\email{scbhoria89@gmail.com}

\author{Pramod Eyyunni}
\address{Pramod Eyyunni, Indian Institute of Science Education and Research Berhampur, Industrial Training Institute (ITI)
 Berhampur, Engineering School Road, Berhampur, Odisha - 760010, India.} 
\email{pramodeyy@gmail.com}

\author{Bibekananda Maji}
\address{Bibekananda Maji, Discipline of Mathematics, Indian Institute of Technology Indore, Simrol, Indore 453552, Madhya Pradesh, India.}
\email{bibek10iitb@gmail.com}

\begin{abstract}
Ramanujan recorded five interesting $q$-series identities in a section that is not as systematically arranged as the other chapters of his 
second notebook. 
These five identities do not seem to have acquired enough attention.
Recently, Dixit and the third author found a one-variable generalization of one of the aforementioned five identities. 
From their generalized identity, they were able to derive the last three of these $q$-series identities, but didn't establish the first two. 
In the present article, we derive a one-variable generalization of the main identity of Dixit and the third author from which we 
successfully deduce all the five $q$-series identities of Ramanujan. In addition to this, 
we also establish a few interesting weighted partition identities from our generalized identity. In the mid $1980$'s, Bressoud and Subbarao found an interesting identity
connecting the generalized divisor function with a weighted partition function, which they proved by means of a purely combinatorial argument. 
Quite surprisingly, we found an analytic proof for a generalization of the identity of Bressoud and Subbarao, 
starting from the fourth identity of the aforementioned five $q$-series identities of Ramanujan. 

\end{abstract}
\maketitle
\section{Introduction}\label{intro} 
Many beautiful $q$-series identities can be found in Ramanujan's notebooks. At the end of the second notebook \cite[pp.~354--355]{ramanujanoriginalnotebook2}, \cite[pp. 302--303]{ramanujantifr}, 
Ramanujan mentioned a list of five $q$-series identities. Before stating these identities, 
we present some notations indispensable to the theory of $q$-series. If
$q$ is a complex number with $|q|<1$, then the $q$-Pochhammer symbol is defined as
\begin{align}
(A)_0 &:=(A;q)_0 =1, \quad (A)_n :=(A;q)_n  = (1-A)(1-Aq)\cdots(1-Aq^{n-1}),
\quad n \geq 1, \nonumber \\
(A)_{\infty} &:=(A;q)_{\i} := \lim_{n \rightarrow \i}(A;q)_n\nonumber.
\end{align}
We now state the list of five identities given by Ramanujan.

{\bf Entry 1:} For $a \neq 0, \ 1-bq^n\neq 0, \ n\geq 1$,
 \begin{equation}\label{entry1}
 \frac{(-aq)_{\infty}}{(bq)_{\infty}}=\sum_{n=0}^{\infty}\frac{(-b/a)_na^nq^{n(n+1)/2}}{(q)_n(bq)_n}.
 \end{equation}
It is also recorded in the Lost Notebook \cite[p.~370]{lnb}. 

{\bf Entry 2:} 
Let $1- a q^n \neq 0$ for $n \geq 1$. Then
 \begin{equation}\label{entry2}
 (aq)_{\infty}\sum_{n=1}^{\infty}\frac{na^nq^{n^2}}{(q)_n(aq)_n}=\sum_{n=1}^{\infty}\frac{(-1)^{n-1}a^nq^{n(n+1)/2}}{1-q^n}.
 \end{equation}

{\bf Entry 3:}
For $a \neq 0,$ and $1-bq^n \neq 0, \ n\geq 0$,
\begin{equation}\label{entry3}
\sum_{n=1}^{\infty} \frac{ (b/a)_n a^n }{ (1- q^n) (b)_n } = \sum_{n=1}^{\infty} \frac{a^n - b^n }{ 1- q^n }. 
\end{equation}
Letting $a \rightarrow 0$, and replacing $b$ by $a q$, one can obtain the next identity.

{\bf Entry 4:}
For $ 1- a q^n \neq 0, \ n \geq 1$,
\begin{equation}\label{entry4}
\sum_{n=1}^{\infty} \frac{(-1)^{n-1} a^n q^{\frac{n(n+1)}{2} } }{(1-q^n) (a q)_n  } = \sum_{n=1}^{\infty} \frac{a^n q^n }{1-q^n}.
\end{equation}
The last identity is the following:

{\bf Entry 5:}
For $ 1- a q^n \neq 0, \ n \geq 0$,
\begin{equation}\label{entry5}
\sum_{n=1}^\infty \frac{a^n (q)_{n-1}}{(1- q^n) (a)_n } = \sum_{n=1}^\infty \frac{n a^n}{1- q^n}.
\end{equation}
These five identities were first proved by Berndt, (see 
\cite[pp.~262--265]{bcbramforthnote} for more information). Infact, the ordering Entries $1$ 
through $5$ given above is due to Berndt.
As is the case with \eqref{entry4}, one can deduce \eqref{entry5} also from \eqref{entry3}. 
Equation \eqref{entry4} was rediscovered independently by Uchimura \cite[Equation (3)]{uchimura81}.
Around the same time as Ramanujan, the special case $a=1$ of \eqref{entry4} 
was obtained by Kluyver \cite{kluyver}:
\begin{equation}\label{Kluyver}
\sum_{n=1}^{\infty} \frac{(-1)^{n-1} q^{\frac{n(n+1)}{2}}}{(1-q^n) ( q)_n  } = \sum_{n=1}^{\infty} \frac{ q^n }{1-q^n}.
\end{equation}
Fine \cite[p.~14, Equations (12.4), (12.42)]{fine} rediscovered this identity, where as Uchimura \cite{uchimura81} found a new expression for the above identity. 
He proved that
\begin{equation}\label{Uchimura}
\sum_{n=1}^\infty n q^n (q^{n+1})_\infty =\sum_{n=1}^{\infty} \frac{(-1)^{n-1} q^{\frac{n(n+1)}{2} } }{(1-q^n) ( q)_n  } = \sum_{n=1}^{\infty} \frac{ q^n }{1-q^n}.
\end{equation}
Bressoud and Subbarao \cite{bresub} derived a beautiful partition theoretic interpretation of this identity. 
Before stating this interpretation, we discuss a few notations that will be useful throughout the paper.
\begin{itemize}
\item $\pi$: an integer partition,

\item $p(n)$: the number of integer partitions of $n$,

\item $ s(\pi):=$ the smallest part of $\pi$,

\item $\ell(\pi):=$ the largest part of $\pi$,

\item $\#(\pi):=$ the number of parts of $\pi$,

\item $\mathrm{rank}(\pi)= \ell(\pi)- \#(\pi)$,

  
\item $\nu_{d}(\pi):=$ the number of parts of $\pi$ without multiplicity,

\item $\nu(j):=$ the number of times the integer $j$ occurs in a partition,

\item $\overline{p}(n):=$ the number of overpartitions of $n$,

\item $\mathcal{P}(n):=$ collection of all integer partitions of $n$,

\item $\mathcal{D}(n):=$ collection of all partitions of $n$ into distinct parts,

\item $\mathcal{P}_{o}(n):=$ collection of all overpartitions of $n$,

\item $\mathcal{P}^{*}(n):=$ partitions into consecutive integers with smallest part $1$.
\end{itemize}

Let $d(n)$ be the number of positive divisors of $n$, then the identity of Bressoud and Subbarao reads
\begin{equation}\label{BS}
\sum_{ \pi \in \mathcal{D}(n)  } (-1)^{ \# (\pi)-1 } s(\pi)=d(n).
\end{equation}
In the same paper \cite{bresub}, using combinatorial arguments, they also derived a more general identity involving the generalized divisor function, 
 \begin{equation}\label{BS-general}
\sum_{ \pi \in \mathcal{D}(n)  } (-1)^{ \# (\pi)-1 } \sum_{j=1}^{s(\pi)} 
( \ell(\pi) - s(\pi) + j ) ^m = \sum_{d | n} d^m,
\end{equation}
for any integer $m \geq 0$. But they didn't mention a generating function identity corresponding to \eqref{BS-general}. 
The identity \eqref{BS} was rediscovered by Fokking, Fokking and Wang \cite{ffw95} and the importance of this identity is that 
 it connects the divisor function and the partition function. 

In 2013, Andrews, Garvan and Liang \cite[Theorem 3.5]{agl13} established a one variable generalization of \eqref{Kluyver}, namely, for $|cq| <1$,  
\begin{equation}\label{ffwz}
\sum_{n=1}^{\infty} \mathrm{FFW}(c,n) q^n = \sum_{n=1}^{\infty} \frac{(-1)^{n-1} q^{\frac{n(n+1)}{2}}}{(1-cq^n) (q)_n } = \frac{1}{1-c} \left(1- \frac{(q)_{\infty} }{(c q)_\infty }  \right), 
\end{equation}
where
\begin{equation*}
\textup{FFW}(c,n):=\sum_{\pi\in\mathcal{D}(n)}(-1)^{\#(\pi)-1}\left(1+c+\cdots+c^{s(\pi)-1}\right).
\end{equation*}

Recently, Dixit and Maji \cite{DM} obtained a one variable generalization of \eqref{entry3} and from this generalized identity,
they were able to connect many well-known $q$-series identities in the literature. Not only that, they also gave partition theoretic 
interpretations of many of the $q$-series identities they found. But they were unable to relate \eqref{entry1} and \eqref{entry2} 
to their generalized identity. Here we state the main identity of Dixit and Maji.
\begin{theorem}\label{gen of Ramanujan's identity}
 Let $a, b, c$ be three complex numbers such that $|a|<1$ and $|cq| < 1$. Then
\begin{equation}\label{entry3gen}
\sum_{n =1}^{\infty} \frac{ (b/a)_n a^n }{ (1- c q^n) (b)_n } =  \sum_{m=0}^{\infty}\frac{(b/c)_mc^m}{(b)_m}\left(\frac{aq^m}{1-aq^m}-\frac{bq^m}{1-bq^m}\right).
\end{equation}
In addition, if $|b|< \min \{|c|, 1\}$,
\begin{equation}\label{entry3gen00}
\sum_{n =1}^{\infty}  \frac{ (b/a)_n a^n }{ (1- c q^n) (b)_n }=\frac{ ( b/c )_{\infty }}{(b)_{\infty} }   \sum_{n=0}^{\infty} \frac{(c)_n ( b/c)^n }{(q)_n}    \sum_{m=1}^{\infty} \frac{a^m - b^m }{1- c q^{m+n}  }.
\end{equation}
\end{theorem}
In the present article, one of our main goals is to derive a one variable generalization of Theorem \ref{gen of Ramanujan's identity}. 
Putting $c = 1$ in \eqref{entry3gen00}, one can easily derive \eqref{entry3}. Again, letting $a \rightarrow 0$ and replacing $b$ by $zq$, 
Dixit and Maji obtained the following one-variable generalization of Andrews, Garvan and Liang's \eqref{ffwz} result, namely, 
for $|cq|<1$,
\begin{equation}\label{gen of Garvan}
\sum_{n=1}^{\infty} \frac{(-1)^{n-1} z^n q^{\frac{n(n+1)}{2} } }{(1-c q^n) (z q)_n  } = \frac{z}{c}\sum_{n=1}^{\infty}\frac{(zq/c)_{n-1}}{(zq)_n}(cq)^n.
\end{equation}
This identity turned out to be a rich source of weighted partition identities. (see \cite[p. 326]{DM}) 
In the next section, we state our one-variable generalization of Theorem \ref{gen of Ramanujan's identity} and its implications, 
which includes a two variable generalization of the result of Andrews, Garvan and Liang
and a beautiful identity of Andrews.
\section{Main Results}
We proceed to state a one-variable generalization of Theorem \ref{gen of Ramanujan's identity}.
\begin{theorem}\label{Generalization of Dixit-Maji}
Let $a, b, c, d$  be four complex numbers such that $|ad|<1$ and $|cq|<1$. Then
\begin{equation}\label{CEM-main identity}
\sum_{n = 1}^{\infty}  \frac{ (b/a)_n (c/d)_n (ad)^n }{ (b)_n (cq)_n } = 
\frac{(a-b)(d-c)}{(ad-b)}\sum_{m=0}^{\infty}\frac{(a)_m (bd/c)_m c^m}{(b)_m (ad)_m}
\left(\frac{adq^m}{1-adq^m}-\frac{bq^m}{1-bq^m}\right).
\end{equation}
\end{theorem}
Substituting $d=1$ in \eqref{CEM-main identity}, one can immediately obtain \eqref{entry3gen}. 
Letting ${a\to 0}$ and $b=zq$ in \eqref{Generalization of Dixit-Maji}, we get a two-variable generalization of the result of 
Andrews, Garvan and Liang, namely, \eqref{ffwz}.
\begin{theorem}\label{form 2 of the generalization}
For $|cq|<1$, we have
\begin{equation}\label{2-var_gen_AGL}
\sum_{n=1}^{\infty}\frac{(-z)^n(c/d)_nd^nq^{n(n+1)/2}}{(zq)_n(cq)_n}=\frac{z(c-d)}{c}\sum_{n=1}^{\infty}\frac{(zdq/c)_{n-1}(cq)^n}{(zq)_n}.
\end{equation}
\end{theorem}
Putting $z=d=1$ in \eqref{2-var_gen_AGL} and employing the $q$-binomial theorem on the right side, one can obtain \eqref{ffwz}. 
Now letting $d \to 0$ in \eqref{2-var_gen_AGL}, we arrive at a beautiful $q$-series identity of Andrews \cite[Corollary 2.2, p. 24]{yesttoday}, namely,
\begin{theorem}\label{2-var_rank generating}
For $ |cq| <1$,  
\begin{equation}\label{d=0analogue_THM2.2_DM}
\sum_{n=1}^{\infty} \frac{ z^n  c^n q^{n^2}}{ (z q)_n (c q)_n } = z \sum_{n=1}^{\infty} \frac{(c q)^n }{(z q)_n}. 
\end{equation}
\end{theorem}
At an initial glance, we see that the left side is symmetric in $z$ and $c$, where as the expression on the right side doesn't seem to be, 
but it is indeed symmetric in $z$ and $c$. Replacing $c$ by $1/z$, one can see that both the sides represent the following 
rank generating function:
\begin{equation}\label{rank_gen_function}
\sum_{n=1}^{\infty} \frac{q^{n^2}}{(zq)_n (z^{-1} q)_n} = \sum_{n=1}^{\infty} \frac{ z (q/z)^n}{ ( z q)_n } =  \sum_{n=1}^{\infty} \sum_{m=-\infty}^{\infty} N(m, n) z^m q^n,
\end{equation}
where $N(m,n)$ denotes the number of partitions of $n$ with rank $m$. 
Theorem \ref{2-var_rank generating} is one of the crucial results which helped us to derive Entry $2$, namely, \eqref{entry2}. 

It would be unfair not to mention the following pleasant generalization of Bressoud-Subbarao's identity \eqref{BS-general}
in the main results. The said generalization is as follows.
\begin{theorem}\label{BS-general one var}
For $m \in \mathbb{Z}$, $n \in \mathbb{N}$ and $a \in \mathbb{C}$ we have,
\begin{equation}\label{One_Var_BS-general}
\sum_{ \pi \in \mathcal{D}(n)  } (-1)^{ \# (\pi)-1 } \sum_{j=1}^{s(\pi)} (\ell(\pi) - s(\pi)  + j ) ^m a^{\ell(\pi) - s(\pi) +j}= \sum_{d | n} d^m a^d.
\end{equation}
\end{theorem}
By letting $a=1$, we immediately derive the identity \eqref{BS-general} of Bressoud and Subbarao.
\begin{remark}
We would like to remark that Bressoud and Subbarao gave a combinatorial proof of their result on the assumption that $m$ is a positive integer. But,
as we have seen, their result actually holds for all integers $m$.
\end{remark}
An analogous result to Bressoud-Subbarao's partition identity may be obtained by plugging $a=-1$ in Theorem \ref{BS-general one var}, yielding 
\begin{corollary} Given an integer $m$ and a positive integer $n$, the following identity holds.
\begin{equation}\label{BSanalogue_a=-1}
\sum_{ \pi \in \mathcal{D}(n) } (-1)^{ \mathrm{rank}(\pi)+s(\pi)-1 } \sum_{j=1}^{s(\pi)} (-1)^j (\ell(\pi)-s(\pi) + j)^m = \sum_{d | n} (-1)^d d^m.
\end{equation}
\end{corollary} 
In the next subsection, we introduce some of the important partition theoretic implications of the main Theorem \ref{Generalization of Dixit-Maji}.
\subsection{Weighted Partition Identities}
In a series of papers, Alladi \cite{alladiantheini,alladitrans, alladiramer} systematically studied many weighted partition identities for the classical partition functions of Euler, Gauss, and Rogers-Ramanujan. 
The study of weighted partition identities has attracted the attention of many mathematicians
\cite{alladi2016},  \cite{berkovichuncujnt}, \cite{dilcher}, 
\cite{garvan1}, \cite{guozeng2015} and \cite{Rgupta},
to name a few. Over the years, mathematicians have found many weighted partition identities similar to \eqref{BS}. Infact, Andrews' famous paper \cite{andrews08} on the $\mathrm{spt}(n)$ 
function was inspired by identity \eqref{BS}. 
Now we mention some of the important weighted partition identities that we obtain from the main identity. 
\begin{theorem}\label{two divisor functions}
Let $n$ be a positive integer and $d_{2, 4}(n)$ be the number of divisors of $n$
that are congruent to $2$ modulo $4$. Then,
\begin{equation}\label{d2,4 identity}
d(n) - 4d_{2, 4}(n) = \sum_{\pi \in \mathcal{P}^{\ast}(n)} (-1)^{\#(\pi) - 1} \omega(\pi),
\end{equation}
where, for $\pi \in \mathcal{P}^{\ast}(n)$, we define  
\begin{equation}\label{omega(pi)}
\omega(\pi) := \nu(\ell(\pi))\prod_{i=1}^{\ell(\pi) - 1} (2\nu(i) - 1).
\end{equation}
\end{theorem}
We demonstrate \eqref{d2,4 identity} with an example, $n=6$.
\begin{center}
\begin{tabular}{|c|c|c|c|c|}
\hline
Partition $\pi \in \mathcal{P}^{\ast}(6)$ & $\#(\pi)$ & $\nu(\ell(\pi))$ & $\omega(\pi)$ & $(-1)^{\#(\pi)-1}\omega(\pi)$ \\
\hline
$3+2+1$ & $3$ & $1$ & $1$ & $1$ \\
\hline
$2+2+1+1$ & $4$ & $2$ & $6$ & $-6$ \\
\hline
$2+1+1+1+1$ & $5$ & $1$ & $7$ & $7$ \\
\hline
$1+1+1+1+1+1$ & $6$ & $6$ & $6$ & $-6$ \\
\hline
\end{tabular} 
\end{center}
So the right side of \eqref{d2,4 identity} for $n=6$ is $1-6+7-6=-4$. Also, $d(6)=4$ and $d_{2, 4}(6) = 2$ so that the 
left side is $4- 4 \cdot 2=-4$, as required.

Corteel and Lovejoy \cite{overp} initiated the study of overpartitions. 
They showed that the overpartition function does satisfy many interesting properties 
similar to the partition function. Here we state a weighted partition representation 
for the number of overpartitons of a given positive integer. 
\begin{proposition}\label{WPI for overpartition}
For each natural number $n$, we have the following identity for $\overline{p}(n)$, the number of overpartitions of $n$,
\begin{equation}\label{overp_omegapi}
\overline{p}(n) = 2\sum_{\pi \in \mathcal{P}^{\ast}(n)} \omega(\pi),
\end{equation} 
where $\omega(\pi)$ is as defined in \eqref{omega(pi)}.
\end{proposition}
We illustrate a particular instance of this identity, namely, for $n=4$.
 \begin{center}
\begin{tabular}{|c|c|c|c|c|}
\hline
Partition $\pi \in \mathcal{P}^{\ast}(4)$ & $\ell(\pi)$ & $\nu(\ell(\pi))$ & $\prod_{i=1}^{\ell(\pi) - 1} (2\nu(i) - 1)$ & $2\omega(\pi)$ \\
\hline
$2+1+1$ & $2$ & $1$ & $3$ & $6$ \\
\hline
$1+1+1+1$ & $1$ & $4$ & $1$ & $8$ \\
\hline
\end{tabular} 
\end{center}
Thus the right hand side of \eqref{overp_omegapi}, for $n=4$, is $6+8=14$, which is indeed equal to the number of overpartitions of $4$.
We state another weighted partition identity which connects the divisor function.
\begin{theorem}\label{analogous to FFW}
Let $\mathcal{D}_1(n)$ represent the set of partitions of $n$ where only the largest part may repeat, 
and $\mathcal{D}^{\ast}(n)$ be the subcollection of $\mathcal{D}_1(n)$ satisfying $\nu(\ell(\pi)) = 2, \ \#(\pi) \geq 3.$ 
Also, $s_2(\pi)$ denotes the second smallest part of $\pi$. Then,
\begin{equation}\label{c=0DM, t=1}
d(n) = 1 + \left\lfloor\frac{n}{2}\right\rfloor - \sum_{\pi \in \mathcal{D}^{\ast}(n)} (-1)^{\#(\pi) - 1}
(s_2(\pi) - s(\pi)).
\end{equation} 
\end{theorem}
We lay out the identity with two examples. Firstly, for $n=9$.
 \begin{center}
\begin{tabular}{|c|c|c|c|}
\hline
Partition $\pi \in \mathcal{D}^{\ast}(9)$ & $\#(\pi)$ & $s_2(\pi) - s(\pi)$& $(-1)^{\#(\pi) - 1} (s_2(\pi) - s(\pi))$ \\
\hline
$4+4+1$ & $3$ & $3$ & $3$ \\
\hline
$3+3+2+1$ & $4$ & $1$ & $-1$ \\
\hline
\end{tabular} 
\end{center}
Thus, the right hand side of \eqref{c=0DM, t=1} for $n=9$ is $1 + \left\lfloor\frac{9}{2}\right\rfloor -2 = 1+4-2=3 = d(9)$, the left hand side. 
Next, for $n=15$, we have
\begin{center}
\begin{tabular}{|c|c|c|c|}
\hline
Partition $\pi \in \mathcal{D}^{\ast}(15)$ & $\#(\pi)$ & $s_2(\pi) - s(\pi)$& $(-1)^{\#(\pi) - 1} (s_2(\pi) - s(\pi))$ \\
\hline
$7+7+1$ & $3$ & $6$ & $6$ \\
\hline
$6+6+3$ & $3$ & $3$ & $3$ \\
\hline
$6+6+2+1$ & $4$ & $1$ & $-1$ \\
\hline
$5+5+4+1$ & $4$ & $3$ & $-3$ \\
\hline
$5+5+3+2$ & $4$ & $1$ & $-1$ \\
\hline
\end{tabular} 
\end{center}
So, the right side of \eqref{c=0DM, t=1} for $n=15$ is $1 + \left\lfloor\frac{15}{2}\right\rfloor - 4 = 1+7-4=4 = d(15)$, as expected.
\begin{remark}
We compare Equation \eqref{c=0DM, t=1}, which is a way of calculating $d(n)$ in terms of certain weighted partitions, 
with Equation \eqref{BS}, which computes $d(n)$ via a sum over partitions into distinct parts. As seen in the above two tables, for $n=9$ and $n=15
$, the sum in the right side of \eqref{c=0DM, t=1} taken over $\mathcal{D}^{\ast}(n)$ involves only $2$ and $5$ partitions respectively. But the number of 
partitions into distinct parts are many more in number, for example, $8$ for $n=9$ and $27$ for $n=15$. So \eqref{c=0DM, t=1} maybe a slightly more
efficient way to compute $d(n)$, as compared to \eqref{BS}. A similar comment holds for Equation \eqref{overp_omegapi}, considering the weighted
sum over partitions in $\mathcal{P}^{\ast}(n)$ maybe an effective way to compute the number of overpartitions $\overline{p}(n)$, which increases even faster than $p(n)$.
\end{remark}

The structure of our paper is as follows. In the next section, we collect a few well-known
results which will be useful throughout this article. In Section \ref{main_result}, we prove our main result 
Theorem \ref{Generalization of Dixit-Maji} and then derive Entry $1$ and Entry $2$ from it.
Section \ref{BS_generalized} is devoted to the proof of Theorem \ref{BS-general one var}, a one-parameter generalization of \eqref{BS-general}. 
In the final section, we prove some weighted partition implications of our main Theorem, including Theorem \ref{two divisor functions},
Proposition \ref{WPI for overpartition} and Theorem \ref{analogous to FFW}.

\section{Preliminary Results}
The $q$-binomial theorem is given by \cite[p.~17, Equation (2.2.1)]{gea1998}, 
\begin{equation}\label{q-binomial}
\sum_{n=0}^{\infty}\frac{(a)_n}{(q)_n}z^n=\frac{(az)_\infty}{(z)_\infty} \quad \text{for} \quad |z| <1.
\end{equation}
Replacing $a$ by $\frac{a}{b}$ and $z$ by $bq$, and letting $b \rightarrow 0$ we obtain another useful version. 
\begin{equation}
\sum_{n=0}^{\infty}\frac{(-a)^nq^{\frac{n(n+1)}{2}}}{(q)_n}=(aq)_\infty  \quad \text{for} \quad |q| <1. \label{alternate form of q-binomial}
\end{equation}
van Hamme \cite{hamme} established the following identity, which is a finite analogue of Kluyver's identity \eqref{Kluyver}. 
For $r \in \mathbb{N}$,
\begin{equation}\label{vanHamme}
\sum_{r=1}^{n}\frac{q^r}{1-q^r} = \sum_{r=1}^{n} 
\left[\begin{matrix} n \\ r \end{matrix}\right] \frac{(-1)^{r-1} q^{\frac{r(r+1)}{2}}}{1-q^r}, 
\end{equation}
where $\left[\begin{matrix} n \\ r \end{matrix}\right]:= \frac{(q)_n}{(q)_{n-r} (q)_r}$ is the 
Gaussian binomial coefficient.
Heine's transformation \cite[p.~359, (III.2)]{GasperRahman} is given by
\begin{align}\label{heine}
{}_{2}\phi_{1}\left( \begin{matrix} a, & b \\
& c \end{matrix} \, ; q, z  \right) = \frac{(\frac{c}{b}, b z; q)_{\infty }}{(c, z; q)_{\infty}} {}_{2}\phi_{1}\left( \begin{matrix} \frac{a b z}{c} , & b \\
& b z \end{matrix} \, ; q, \frac{c}{b}  \right).
\end{align}
Andrews  \cite[p.~252, Theorem 2.1]{andpar}, \cite[p.~263]{abramlostI} showed that, for any $N \in \mathbb{N}$,  
\begin{equation}\label{frgfbil}
\sum_{n=0}^{N}\left[\begin{matrix} N \\ n \end{matrix}\right]\frac{(q)_nq^{n^2}}{(zq)_n(z^{-1}q)_n}=\frac{1}{(q)_N}+(1-z)\sum_{n=1}^{N}\left[\begin{matrix} N \\ n \end{matrix}\right]\frac{(-1)^n(q)_nq^{n(3n+1)/2}}{(q)_{n+N}}\left(\frac{1}{1-zq^n}-\frac{1}{z-q^n}\right).
\end{equation}
The left side is a finite analogue of the rank generating function \eqref{rank_gen_function}. 
Recently, Dixit et al. \cite[p.~8]{DEMS} found a nice interpretation of \eqref{frgfbil} in terms of vector partitions 
(see \cite{DEMS}, \cite{EMS} for finite analogues of rank and crank generating functions).
By the substitution $z=1$ in \eqref{frgfbil}, we obtain
\begin{equation}\label{andrews_z=1}
\frac{1}{(q)_N}= \sum_{n=0}^{N}\left[\begin{matrix} N \\ n \end{matrix}\right]\frac{q^{n^2}}{(q)_n}. 
\end{equation}

\section{Proof of the main result and its applications}\label{main_result}

\begin{proof}[Theorem \textup{\ref{Generalization of Dixit-Maji}}][]
We shall start the proof by recalling a ${}_{3}\phi_{2}$ transformation formula in \cite[p.~359, (III.9)]{GasperRahman}:
\begin{align*}
{}_{3}\phi_{2}\left[\begin{matrix} A, & B, & C \\ & D, & E \end{matrix} \, ; q, \frac{DE}{ABC}  \right] = \frac{\big(\frac{E}{A}\big)_\infty\big( \frac{DE}{BC}\big)_{\infty }}{\big(E\big)_\infty \big(\frac{DE}{ABC}\big)_{\infty}} 
{}_{3}\phi_{2}\left[\begin{matrix}A,& \frac{D}{B} , & \frac{D}{C} \\
& D,& \frac{DE}{BC} \end{matrix} \, ; q, \frac{E}{A} \right].
\end{align*}
Setting $A=q, \ B=\frac{bq}{a}, \ C=\frac{cq}{d}, \ D=bq, \ E=cq^2$ above, we get,
\begin{align*}
{}_{3}\phi_{2}\left[ \begin{matrix} q, & \frac{bq}{a} , & \frac{cq}{d} \\ & bq, & cq^2 \end{matrix} \, ; q, ad \right] = \frac{(cq)_\infty(adq)_{\infty }}{(cq^2)_\infty (ad)_{\infty}} 
{}_{3}\phi_{2}\left[ \begin{matrix}q,& a & \frac{bd}{c} \\
& bq,& adq \end{matrix} \, ; q, cq  \right].
\end{align*}

In other words,
\begin{align*}
\sum_{n = 0}^{\infty} \frac{ (q)_n(bq/a)_n (cq/d)_n (ad)^n }{ (bq)_n (cq^2)_n (q)_n } &=\frac{(1-cq)}{(1-ad)}
\sum_{m=0}^{\infty}\frac{(q)_m(a)_m (bd/c)_m (cq)^m}{(bq)_m (adq)_m(q)_m} \\
\Rightarrow \sum_{n = 1}^{\infty} \frac{(bq/a)_{n-1} (cq/d)_{n-1} (ad)^{n-1} }{ (bq)_{n-1} (cq^2)_{n-1}} &=\frac{(1-cq)}{(1-ad)}\sum_{m=0}^{\infty}
\frac{(a)_m (bd/c)_m (cq)^m}{(bq)_m (adq)_m}.
\end{align*}

Multiplying by $\frac{(1-b/a)(1-c/d)ad}{(1-b)(1-cq)}$ on both sides, we get
\begin{align*}
\sum_{n = 1}^{\infty} \frac{(b/a)_n (c/d)_n (ad)^n }{ (b)_n (cq)_n} &=\frac{(1-b/a)(1-c/d)ad}{(1-b)(1-ad)}\sum_{m=0}^{\infty}\frac{(a)_m (bd/c)_m (cq)^m}{(bq)_m (adq)_m}\\
\Rightarrow \sum_{n = 1}^{\infty} \frac{(b/a)_n (c/d)_n (ad)^n }{ (b)_n (cq)_n} &=\frac{(a-b)(d-c)}{(1-b)(1-ad)}\sum_{m=0}^{\infty}\frac{(a)_m (bd/c)_m (cq)^m}{(bq)_m (adq)_m}.
\end{align*}

Multiplying and dividing by $(ad-b)$ on the right side above, we arrive at
\begin{align*}
\sum_{n = 1}^{\infty} \frac{(b/a)_n (c/d)_n (ad)^n }{ (b)_n (cq)_n} &=\frac{(a-b)(d-c)}{(ad-b)}\sum_{m=0}^{\infty}\frac{(a)_m (bd/c)_m (cq)^m(ad-b)}{(b)_{m+1} (ad)_{m+1}}\\
\Rightarrow \sum_{n = 1}^{\infty} \frac{(b/a)_n (c/d)_n (ad)^n }{ (b)_n (cq)_n} &=\frac{(a-b)(d-c)}{(ad-b)}\sum_{m=0}^{\infty}\frac{(a)_m (bd/c)_m c^m}{(b)_m (ad)_m }\frac{(ad-b)q^m}{(1-adq^m)(1-bq^m)}.
\end{align*}

This is nothing but what we set out to prove, namely,
\begin{align*} 
\sum_{n = 1}^{\infty}  \frac{ (b/a)_n (c/d)_n (ad)^n }{ (b)_n (cq)_n } = 
\frac{(a-b)(d-c)}{(ad-b)}\sum_{m=0}^{\infty}\frac{(a)_m (bd/c)_m c^m}{(b)_m (ad)_m}
\left(\frac{adq^m}{1-adq^m}-\frac{bq^m}{1-bq^m}\right).
\end{align*}
\end{proof}

\subsection{Proof of Entry 1:}
 To prove Entry $1$ (Equation \eqref{entry1}), we let $d\rightarrow 0$ in \eqref{CEM-main identity} which yields 
\begin{equation*}
 \sum_{n=1}^{\infty} \frac{\left( \frac{b}{a}\right)_n (-ac)^n q^{\frac{n(n-1)}{2}}}{(b)_n (cq)_n}
= \frac{(b-a)}{q} \sum_{n=1}^{\infty} \frac{(a)_{n-1} (cq)^n}{(b)_n}.
\end{equation*}
Multiply by $\frac{1-b}{c(a-b)}$ throughout to obtain
\begin{align*}
\sum_{n=1}^{\infty}\frac{(bq/a)_{n-1}(-ac)^{n-1}q^{n(n-1)/2}}{(bq)_{n-1}(cq)_n}=\sum_{n=1}^{\infty}\frac{(a)_{n-1}(cq)^{n-1}}{(bq)_{n-1}}\\
\Rightarrow \frac{1}{1-cq}\sum_{n=0}^{\infty}\frac{(bq/a)_n(-ac)^nq^{n(n+1)/2}}{(bq)_n(cq^2)_n}=\sum_{n=0}^{\infty}\frac{(a)_n(cq)^n}{(bq)_n}.
\end{align*}
Substituting $b=1$ and invoking the well-known $q$-binomial theorem \eqref{q-binomial} on the right side, we have
\begin{align*}
\frac{1}{1-cq}\sum_{n=0}^{\infty}\frac{(q/a)_n(-ac)^nq^{n(n+1)/2}}{(q)_n(cq^2)_n}=\sum_{n=0}^{\infty}\frac{(a)_n(cq)^n}{(q)_n}=\frac{(acq)_\infty}{(cq)_\infty}.
\end{align*}
Now putting $c = \frac{b}{q}$ and then canceling the factor $\frac{1}{1-b}$ on both the sides, we arrive at
\begin{align*}
\sum_{n=0}^{\infty}\frac{(q/a)_n(-ab/q)^nq^{n(n+1)/2}}{(q)_n(bq)_n}=\frac{(ab)_\infty}{(bq)_\infty}.
\end{align*}
Finally we replace $a$ by $\frac{-aq}{b}$ to attain the desired result, namely, Ramanujan's Entry $1$,
Equation \eqref{entry1}.

\subsection{Alternate proof of Entry 1} 
To obtain an alternate proof, we shall substitute $z=1$ in \eqref{2-var_gen_AGL} and then replace $d$ by $-a$ and $c$ by $b$ to arrive at 
\begin{align*}
\sum_{n=1}^{\infty}\frac{(-b/a)_n a^n q^{n(n+1)/2}}{(q)_n(bq)_n}&=\frac{(b+a)}{b}\sum_{n=1}^{\infty}\frac{(-aq/b)_{n-1}(bq)^n}{(q)_n}. \\ 
&=\sum_{n=1}^{\infty}\frac{(-a/b)_n(bq)^n}{(q)_n}.
\end{align*}
Upon adding $1$ on both sides and using \eqref{q-binomial} on the right side, we end up with \eqref{entry1}.

\subsection{Proof of Entry 2} We start by differentiating \eqref{d=0analogue_THM2.2_DM} with respect to $z$, 
\begin{equation}\label{DIFFwrt_z}
\sum_{n=1}^{\infty} \frac{z^{n-1}c^n q^{n^2}}{(zq)_n(cq)_n}\left\{n+ \sum_{r=1}^{n} 
\frac{zq^r}{1-zq^r}\right\}=\sum_{n=1}^{\infty} \frac{(cq)^n}{(zq)_n} \left\{1+\sum_{r=1}^{n} \frac{zq^r}{1-zq^r}\right\}. 
\end{equation}

Set $z=1$ above and separating the terms into two sums on both the sides, we get 
\begin{equation}\label{DIFFwrt_z_z=1}
\sum_{n=1}^{\infty} \frac{nc^n q^{n^2}}{(q)_n(cq)_n}+\sum_{n=1}^{\infty} \frac{c^n q^{n^2}}{(q)_n (cq)_n} \sum_{r=1}^{n} 
\frac{q^r}{1-q^r} =\sum_{n=1}^{\infty} \frac{(cq)^n}{(q)_n} + \sum_{n=1}^{\infty} \frac{(cq)^n}{(q)_n}
\sum_{r=1}^{n} \frac{q^r}{1-q^r}.
\end{equation}
For convenience, we consider each of the sides one by one in the above equation, starting with the left hand side. For the inner sum of the second term,
we employ \eqref{vanHamme} to obtain
\begin{align*}
&\sum_{n=1}^{\infty} \frac{nc^n q^{n^2}}{(q)_n(cq)_n}+\sum_{n=1}^{\infty} \frac{c^n q^{n^2}}{(q)_n (cq)_n}  \sum_{r=1}^{n} \frac{(q)_n}{(q)_{n-r} (q)_r} \frac{(-1)^{r-1} q^{\frac{r(r+1)}{2}}}{1-q^r} \\
& =\sum_{n=1}^{\infty} \frac{nc^n q^{n^2}}{(q)_n(cq)_n}+\sum_{r=1}^{\infty} \frac{(-1)^{r-1} q^{\frac{r(r+1)}{2}}}{(q)_r (1-q^r)}
\sum_{n=r}^{\infty} \frac{c^n q^{n^2}}{(cq)_n (q)_{n-r}} \\
&= \sum_{n=1}^{\infty} \frac{nc^n q^{n^2}}{(q)_n(cq)_n}+\sum_{r=1}^{\infty} \frac{(-1)^{r-1} q^{\frac{r(r+1)}{2}}}{(q)_r (1-q^r)}
\sum_{t=0}^{\infty} \frac{c^{r+t} q^{(r+t)^2}}{(cq)_{r+t} (q)_t} \\
& =\sum_{n=1}^{\infty} \frac{nc^n q^{n^2}}{(q)_n(cq)_n}+\sum_{r=1}^{\infty} \frac{(-1)^{r-1} c^r 
q^{r^2 + \frac{r(r+1)}{2}}}{(q)_r (1-q^r)} \sum_{t=0}^{\infty} 
\frac{c^t q^{t(t+2r)}}{(cq)_{r+t} (q)_t}.
\end{align*}
Using $(cq)_{r+t} = (cq)_r (cq^{r+1})_t$, we can write the previous line as
\begin{align}
&\sum_{n=1}^{\infty} \frac{nc^n q^{n^2}}{(q)_n(cq)_n}+ \sum_{r=1}^{\infty} \frac{(-1)^{r-1} c^r 
q^{r^2 + \frac{r(r+1)}{2}}}{(q)_r (cq)_r (1-q^r)} \sum_{t=0}^{\infty} 
\frac{c^t q^{t(t+2r)}}{(cq^{r+1})_{t} (q)_t} \nonumber \\
&= \sum_{n=1}^{\infty} \frac{nc^n q^{n^2}}{(q)_n(cq)_n}+\sum_{r=1}^{\infty} \frac{(-1)^{r-1} c^r 
q^{r^2 + \frac{r(r+1)}{2}}}{(q)_r (cq)_r (1-q^r)}\sum_{t=0}^{\infty} \frac{q^{t^2} (cq^{2r})^t}{(cq^{r+1})_t (q)_t}. \label{simplpostvH}
\end{align}
We now express the inner sum inside the second summation in a suitable form so as to apply one of 
Heine's transformations. We first note the following series of steps:
\begin{align}
   _2 \phi_1 \left[ \frac{q}{f}, \ \frac{q}{g}, \ h; \ q; \ fgw \right] &= 
  \sum_{t=0}^{\infty} \frac{\left(\displaystyle\frac{q}{f}\right)_t 
  \left(\displaystyle\frac{q}{g}\right)_t }{(h)_t (q)_t}(fgw)^t \label{applying_Heine} \\
  &= \sum_{t=0}^{\infty} \frac{q^{t^2 + t}}{(h)_t (q)_t}w^t \nonumber,
\end{align} 
upon letting $f$ and $g$ tend to zero. We now make the substitutions 
$w=cq^{2r-1}, \ h=cq^{r+1}$ in the last sum above to get
\begin{align*}
 \sum_{t=0}^{\infty} \frac{q^{t^2} (cq^{2r})^t}{(cq^{r+1})_t (q)_t},
\end{align*}
which is precisely the inner sum in \eqref{simplpostvH}. Applying Heine's transformation \eqref{heine}
to \eqref{applying_Heine}, we obtain
\begin{equation*}
 \sum_{t=0}^{\infty} \frac{\left(\displaystyle\frac{q}{f}\right)_t 
 \left(\displaystyle\frac{q}{g}\right)_t}{(h)_t (q)_t}(fgw)^t = 
 \frac{\left( \displaystyle\frac{gh}{q}\right)_{\infty} (fqw)_{\infty}}{(h)_{\infty} (fgw)_{\infty}}
 \sum_{t=0}^{\infty} \frac{ \left( \displaystyle\frac{q^2 w}{h}\right)_t 
 \left( \displaystyle\frac{q}{g}\right)_t}{(fqw)_t (q)_t} \left( \frac{gh}{q}\right)^t.
 \end{equation*}
As we had done with the left side here, we let $f, g \rightarrow 0$ and then do 
the substitutions $w=cq^{2r-1}, \ h=cq^{r+1}$ on the right hand side, to get the identity
\begin{equation}\label{madhya}
 \sum_{t=0}^{\infty} \frac{q^{t^2} (cq^{2r})^t}{(cq^{r+1})_t (q)_t} = 
\frac{1}{(cq^{r+1})_{\infty}}\sum_{t=0}^{\infty}\frac{(q^r)_t (-c)^t q^{rt+\frac{t(t+1)}{2}}}{(q)_t}.
\end{equation}
Utilizing \eqref{madhya} in \eqref{simplpostvH}, we obtain
\begin{align*}
&\sum_{n=1}^{\infty} \frac{nc^n q^{n^2}}{(q)_n(cq)_n}+\sum_{r=1}^{\infty} \frac{(-1)^{r-1} c^r 
q^{r^2 + \frac{r(r+1)}{2}}}{(q)_r (cq)_r (1-q^r)}\frac{1}{(cq^{r+1})_{\infty}}
\sum_{t=0}^{\infty}\frac{(q^r)_t (-c)^t q^{rt+\frac{t(t+1)}{2}}}{(q)_t}\\
&=\sum_{n=1}^{\infty} \frac{nc^n q^{n^2}}{(q)_n(cq)_n}+\frac{1}{(cq)_{\infty}}\sum_{r=1}^{\infty} \frac{(-1)^{r-1} c^r q^{r^2 + \frac{r(r+1)}{2}}}{(q)_r (1-q^r)}\sum_{t=0}^{\infty}\left[\begin{matrix} t+r-1 \\ t \end{matrix}\right](-c)^t q^{rt+\frac{t(t+1)}{2}}
\\&=\sum_{n=1}^{\infty} \frac{nc^n q^{n^2}}{(q)_n(cq)_n}-\frac{1}{(cq)_{\infty}}\sum_{r=1}^{\infty} \frac{q^{r^2}}{(q)_r (1-q^r)}\sum_{t=0}^{\infty}\left[\begin{matrix} t+r-1 \\ t \end{matrix}\right](-c)^{t+r}q^{\frac{(t+r)(t+r+1)}{2}}.
\end{align*}
Substituting $t+r=s$ in the above summation, we have
\begin{align}
&\sum_{n=1}^{\infty} \frac{nc^n q^{n^2}}{(q)_n(cq)_n}-\frac{1}{(cq)_{\infty}}\sum_{r=1}^{\infty} \frac{ q^{r^2}}{(q)_r (1-q^r)}\sum_{s=r}^{\infty}\left[\begin{matrix} s-1 \\ r-1 \end{matrix}\right](-c)^sq^{\frac{s(s+1)}{2}}\nonumber
\\&=\sum_{n=1}^{\infty} \frac{nc^n q^{n^2}}{(q)_n(cq)_n}-\frac{1}{(cq)_{\infty}}\sum_{s=1}^{\infty}\frac{(q)_{s-1}(-c)^sq^{\frac{s(s+1)}{2}}}{(q)_s}\sum_{r=1}^{s}\left[\begin{matrix} s \\ r \end{matrix}\right]\frac{q^{r^2}}{(q)_r}\nonumber
\\&=\sum_{n=1}^{\infty} \frac{nc^n q^{n^2}}{(q)_n(cq)_n}-\frac{1}{(cq)_{\infty}}\sum_{s=1}^{\infty}\frac{(-c)^sq^{\frac{s(s+1)}{2}}}{1-q^s}\left(\frac{1}{(q)_s}-1\right)\label{Simplified_LHS_DIFFwrt_z_z=1},
\end{align}
the simplification in the inner sum resulting from \eqref{andrews_z=1}. So the left side of \eqref{DIFFwrt_z_z=1} now takes the form of Equation 
\eqref{Simplified_LHS_DIFFwrt_z_z=1}. Coming to the right side of \eqref{DIFFwrt_z_z=1}, we begin by applying van Hamme's identity \eqref{vanHamme} 
for the inner sum of the second term:
\begin{align*}
 &\sum_{n=1}^{\infty} \frac{(cq)^n}{(q)_n} + \sum_{n=1}^{\infty} \frac{(cq)^n}{(q)_n}
\sum_{r=1}^{n} \frac{q^r}{1-q^r} \\
&= \left\{ \frac{1}{(cq)_{\infty}} - 1\right\} + \sum_{n=1}^{\infty}\frac{(cq)^n}{(q)_n}
\sum_{r=1}^{n} \left[\begin{matrix} n \\ r \end{matrix}\right] \frac{(-1)^{r-1} q^{\frac{r(r+1)}{2}}}{1-q^r},
\end{align*}
the first expression in the parentheses resulting from an application of \eqref{q-binomial}. Proceeding further, we get
\begin{align}
&\left\{ \frac{1}{(cq)_{\infty}} - 1\right\} + \sum_{r=1}^{\infty} \frac{(-1)^{r-1} q^{\frac{r(r+1)}{2}}}{(q)_r (1-q^r)}
\sum_{n=r}^{\infty} \frac{c^n q^n}{(q)_{n-r}} \nonumber \\
&= \left\{ \frac{1}{(cq)_{\infty}} - 1\right\} + \sum_{r=1}^{\infty} \frac{(-1)^{r-1} 
(cq)^r q^{\frac{r(r+1)}{2}}}{(q)_r (1-q^r)} \sum_{s=0}^{\infty} \frac{(cq)^s}{(q)_s}\nonumber \\
&=\left\{ \frac{1}{(cq)_{\infty}} - 1\right\} +\frac{1}{(cq)_{\infty}} \sum_{n=1}^{\infty} \frac{(-1)^{n-1} 
(cq)^n q^{\frac{n(n+1)}{2}}}{(q)_n (1-q^n)}, \label{Simplified_RHS_DIFFwrt_z_z=1}  
\end{align}
by invoking \eqref{q-binomial} in the innermost sum. 

Equating the simplified versions of both sides of \eqref{DIFFwrt_z_z=1}, namely, \eqref{Simplified_LHS_DIFFwrt_z_z=1} and 
\eqref{Simplified_RHS_DIFFwrt_z_z=1} together, one arrives at
\begin{align*}
\sum_{n=1}^{\infty} \frac{nc^nq^{n^2} }{(q)_n(cq)_n}-\frac{1}{(cq)_{\infty}}\sum_{n=1}^{\infty}\frac{(-c)^nq^{\frac{n(n+1)}{2}}}{(q)_n(1-q^n)}+  \frac{1}{(cq)_{\infty}}\sum_{n=1}^{\infty}\frac{(-c)^nq^{\frac{n(n+1)}{2}}}{1-q^n}\\
=\left\{ \frac{1}{(cq)_{\infty}} - 1\right\}+\frac{1}{(cq)_{\infty}} \sum_{n=1}^{\infty} \frac{(-1)^{n-1} 
(cq)^n q^{\frac{n(n+1)}{2}}}{(q)_n (1-q^n)}. 
\end{align*}
Now a nice simplification happens if we combine the second term of the left side with the second term of the right side giving, after rearrangement,
\begin{align*}
\sum_{n=1}^{\infty} \frac{nc^nq^{n^2} }{(q)_n(cq)_n} - \frac{1}{(cq)_{\infty}}\sum_{n=1}^{\infty}\frac{(-c)^nq^{\frac{n(n+1)}{2}}}{(q)_n(1-q^n)}(1-q^n) 
=\left\{ \frac{1}{(cq)_{\infty}} - 1\right\} +
\frac{1}{(cq)_{\infty}}\sum_{n=1}^{\infty}\frac{(-1)^{n-1} c^nq^{\frac{n(n+1)}{2}}}{1-q^n}.
\end{align*}
Upon canceling the $(1-q^n)$ factor in the second term of the left side and utilizing \eqref{alternate form of q-binomial} yields 
\begin{align*}
\sum_{n=1}^{\infty} \frac{nc^nq^{n^2} }{(q)_n(cq)_n} - \frac{1}{(cq)_{\infty}}\Big\{(cq)_{\infty}-1\Big\}
=\left\{ \frac{1}{(cq)_{\infty}} - 1\right\} + \frac{1}{(cq)_{\infty}}\sum_{n=1}^{\infty}\frac{(-1)^{n-1} c^n q^{\frac{n(n+1)}{2}}}{1-q^n}.
\end{align*}
Canceling $\frac{1}{(cq)_{\infty}} - 1$ from both sides, the surviving terms leave us with Entry 2
of Ramanujan.

\subsection{Alternate proof of Entry 2} An alternate proof of Ramanujan's Entry $2$ can also be obtained by 
differentiating the identity \eqref{d=0analogue_THM2.2_DM} with respect to $c$.
In the previous proof, we have differentiated \eqref{d=0analogue_THM2.2_DM} with respect to $z$ and obtained \eqref{DIFFwrt_z}. 
We had also noted earlier that the left side of \eqref{d=0analogue_THM2.2_DM} is
symmetric in $z$ and $c$. So, even upon differentiation with respect to $c$ and letting 
$c \rightarrow 1$, the left side of \eqref{d=0analogue_THM2.2_DM} would give us the same expression as in the previous proof, namely, 
the left side of \eqref{DIFFwrt_z_z=1} (or the simplified \eqref{Simplified_LHS_DIFFwrt_z_z=1}), with $z$ replacing $c$ here. That is,
\begin{equation}\label{Simplified_LHS_DIFFwrt_c=1}
\sum_{n=1}^{\infty} \frac{nz^n q^{n^2}}{(q)_n(zq)_n}-\frac{1}{(zq)_{\infty}}\sum_{s=1}^{\infty}\frac{(-z)^sq^{\frac{s(s+1)}{2}}}{1-q^s}
\left(\frac{1}{(q)_s}-1\right). 
\end{equation}
So it would be interesting to differentiate the right side of \eqref{d=0analogue_THM2.2_DM} too, with respect to $c$. 
We do this and then set $c=1$ to get 
\begin{equation}\label{Simplified_RHS_DIFFwrt_c_c=1}
 \sum_{n=1}^{\infty} \frac{nzq^n}{(zq)_n}.
\end{equation}
Equating \eqref{Simplified_LHS_DIFFwrt_c=1} and \eqref{Simplified_RHS_DIFFwrt_c_c=1}, we get after a little rearrangement,
\begin{equation}\label{Simplified LHS=RHS of DIFFwrt_c_c=1}
     \frac{-1}{(zq)_{\infty}}\sum_{n=1}^{\infty}\frac{(-z)^nq^{\frac{n(n+1)}{2}}}{(q)_n(1-q^n)}+  \frac{1}{(zq)_{\infty}}\sum_{n=1}^{\infty}\frac{(-z)^nq^{\frac{n(n+1)}{2}}}{1-q^n} = \sum_{n=1}^{\infty} \frac{nzq^n }{(zq)_n} 
- \sum_{n=1}^{\infty} \frac{nz^nq^{n^2} }{(q)_n(zq)_n}.
\end{equation}
Note that our aim is to prove Ramanujan's Entry 2, namely, Equation \ref{entry2},
\begin{align*}
\frac{1}{(zq)_{\infty}}\sum_{n=1}^{\infty}\frac{(-1)^{n-1}z^nq^{\frac{n(n+1)}{2}}}{1-q^n} = \sum_{n=1}^{\infty} \frac{nz^nq^{n^2} }{(q)_n(zq)_n}.
\end{align*}
So, in Equation \eqref{Simplified LHS=RHS of DIFFwrt_c_c=1}, if we are able to show that
\begin{align}\label{Gerneralization of uchimura identity form 2}
\frac{-1}{(zq)_{\infty}}\sum_{n=1}^{\infty}\frac{(-z)^nq^{\frac{n(n+1)}{2}}}{(q)_n(1-q^n)}=\sum_{n=1}^{\infty} \frac{nzq^n }{(zq)_n},
\end{align}
then we will be through.
To do this, we first multiply both sides of 
\eqref{Gerneralization of uchimura identity form 2} by $\displaystyle\frac{(zq)_{\infty}}{z}$, so as 
to put the equation in an equivalent form which is neater,

\begin{equation}\label{one variable Gerneralization of uchimura identity}
\sum_{n=1}^{\infty}\frac{(-z)^{n-1}q^{\frac{n(n+1)}{2}}}{(q)_n(1-q^n)}=\sum_{n=1}^{\infty} nq^n(zq^{n+1})_{\infty}.
\end{equation}
We now prove \eqref{one variable Gerneralization of uchimura identity} by showing that their partition-theoretic interpretations are the same. 
The left side of \eqref{one variable Gerneralization of uchimura identity} can be written as follows,
\begin{align*}
\sum_{n=1}^{\infty}\frac{(-z)^{n-1}q^{\frac{n(n+1)}{2}}}{(q)_n(1-q^n)}=\sum_{n=1}^{\infty} (-z)^{n-1}
\frac{q^1}{1-q^1} \cdot \frac{q^2}{1-q^2} \cdots \frac{q^{n-1}}{1-q^{n-1}}\cdot\frac{q^n}{(1-q^n)^2}.
\end{align*}
For $1\leq r\leq n-1,
\frac{q^r}{1-q^r}=\sum_{k_r=1}^{\infty}q^{k_r \cdot r}$ and $\frac{q^n}{(1-q^n)^2}=\sum_{k_n=1}^{\infty} k_n q^{k_n \cdot n}$. 
Therefore the left hand side of \eqref{one variable Gerneralization of uchimura identity} 
becomes 
\begin{align*}
\sum_{m=1}^{\infty}\left(\sum_{\pi\in\mathcal{P}^*(m)}(-z)^{\ell(\pi)-1}\nu(\ell(\pi))\right)q^m,
\end{align*}
which upon conjugation yields
\begin{align*}
\sum_{m=1}^{\infty}\left(\sum_{\pi\in\mathcal{D}(m)}(-z)^{\#(\pi)-1}s(\pi)\right)q^m.
\end{align*}
Now, the generating function for this weighted partition over distinct part partitions can be written as 
\begin{align*}
\sum_{n=1}^{\infty} nq^n(1-zq^{n+1})(1-zq^{n+2})\cdots=\sum_{n=1}^{\infty}nq^n(zq^{n+1})_{\infty}.
\end{align*}
Hence \eqref{one variable Gerneralization of uchimura identity} holds and we have proved Entry $2$. 
\begin{remark}
In the course of giving this alternate proof, we obtained Equation 
\eqref{one variable Gerneralization of uchimura identity}, which is a one-variable
generalization of the first equality in \eqref{Uchimura}. It would be interesting to find 
a direct analytic proof of \eqref{one variable Gerneralization of uchimura identity}, and also
generalize the second equality in \eqref{Uchimura}.
\end{remark}

\section{Proof of Bressoud-Subbarao's partition identity}\label{BS_generalized}
The identities of Ramanujan continue to unearth many mathematical treasures hidden under them. 
A glimpse of one such instance may be seen in the present section wherein we found that Bressoud-Subbarao's identity \eqref{BS-general} 
is nothing more than an application of successive differentiation to the partition interpretation of Entry $4$ (Equation \eqref{Part_Imp_Entry4} below). 
Indeed, using a couple of simple operators, we prove a generalization of Bressoud-Subbarao's partition identity, namely, Theorem \ref{BS-general one var}.
\begin{definition}
Let $f(a)$ be a polynomial in `$a$'. We define the operators `$D$' and `$I$' respectively by 
$D[f(a)]:=a \displaystyle \frac{\partial}{\partial a}\{f(a)\}$, and
$I[f(a)]:= \displaystyle \int_{0}^{a}\frac{f(t)}{t}dt$.
\end{definition}
We begin with a few crucial lemmas. 
\begin{lemma} Entry $4$ has an appealing weighted partition implication, namely, 
\begin{equation}\label{Part_Imp_Entry4}
\sum_{\pi\in \mathcal{D}(n)}(-1)^{\#(\pi)-1}a^{\ell(\pi)-s(\pi)+1}\left(\frac{a^{s(\pi)}-1}{a-1}\right)=\sum_{d|n}a^d.
\end{equation} 
\end{lemma}
We would like to remark here that \eqref{Part_Imp_Entry4} is Corollary $2.6$ of Dixit and Maji, see \cite{DM}.
\begin{lemma}\label{D^m lemma for BS-general}
 Let $k \in \mathbb{N}, \ r \in \mathbb{N} \cup \{0\}$ and let $ F_0(a):=a^r(a+a^2+\cdots+a^k)$. Then, for each $m \in \mathbb{N}$,
 \begin{equation}\label{Ind_F_0(a)}
D^m \left[F_0(a)\right] = (r+1)^m a^{r+1}+(r+2)^m a^{r+2}+\cdots+(r+k)^m a^{r+k} = 
\sum_{j=1}^{k} (r+j)^m a^{r+j}. 
\end{equation}
\end{lemma}
\begin{proof}
We have
\begin{align*}
D^1 \left[F_0(a)\right] = D \left[F_0(a)\right] &= a\frac{\partial }{\partial a}\left\{F_0(a)\right\}=
a\frac{\partial }{\partial a}\left\{a^r(a+a^2+ \cdots +a^k)\right\}\\
&= a \sum_{j=1}^{k} \frac{\partial }{\partial a} \{a^{r+j}\}
=\sum_{j=1}^{k} (r+j) a^{r+j},
\end{align*}
so that \eqref{Ind_F_0(a)} holds for $m=1$. 
Suppose that Equation \eqref{Ind_F_0(a)} holds for some integer $p \geq 1$, that is,
\begin{equation*}\label{Induction_p}
D^p \left[F_0(a)\right] = \sum_{j=1}^{k} (r+j)^p a^{r+j}.
\end{equation*}
Now, 
\begin{align*}
D^{p+1} \left[F_0(a)\right]&=D\left\{D^p \left[F_0(a)\right]\right\} 
=a\frac{\partial }{\partial a}\left\{\sum_{j=1}^{k} (r+j)^p a^{r+j}\right\}\\
&= a\sum_{j=1}^{k} \frac{\partial }{\partial a}\left\{(r+j)^p a^{r+j}\right\}
= \sum_{j=1}^{k} (r+j)^{p+1} a^{r+j}.
\end{align*}
This finishes the proof of the lemma.
\end{proof}

Similarly, we have the following result for the operator $I$.
\begin{lemma}\label{I^m lemma for BS-general}
Let $k, r \in \mathbb{N} \cup \{0\}$ and let $ F_0(a):=a^{r}(a+a^2+\cdots+a^k)$. Given any $m \in \mathbb{N}$, 
we have
\begin{equation}\label{Ioperator_Ind_F_0(a)}
I^m \left[F_0(a)\right] = \frac{a^{r+1}}{(r+1)^m} + \frac{a^{r+2}}{(r+2)^m} +
\cdots+\frac{a^{r+k}}{(r+k)^m} = \sum_{j=1}^{k} 
\frac{a^{r+j}}{(r+j)^m}. 
\end{equation}
\end{lemma}
\begin{proof}
Firstly, consider
\begin{align*}
I^1 \left[F_0(a)\right] = I \left[F_0(a)\right] &= 
\int_{0}^{a}\frac{F_0(t)}{t}= \int_{0}^{a} t^r (1 + t + \cdots + t^{k-1})dt \\
&= \sum_{j=0}^{k-1} \int_{0}^{a} t^{r+j} dt =  \sum_{j=1}^{k} \frac{a^{r+j}}{r+j},
\end{align*}
which means that \eqref{Ioperator_Ind_F_0(a)} is true for $m=1$. 
We now assume that Equation \eqref{Ioperator_Ind_F_0(a)} holds for some integer $p \geq 1$, 
thus,
\begin{equation*}\label{I_Induction_p}
I^p \left[F_0(a)\right] = \sum_{j=1}^{k} \frac{a^{r+j}}{(r+j)^p}.
\end{equation*}
Now, 
\begin{align*}
I^{p+1} \left[F_0(a)\right]=I\left\{I^p \left[F_0(a)\right]\right\} 
&=\int_{0}^{a} \frac{I^p\left[F_0(t)\right]}{t} dt
= \int_{0}^{a} \sum_{j=1}^{k} \frac{t^{r+j-1}}{(r+j)^p}
= \sum_{j=1}^{k} \frac{a^{r+j}}{(r+j)^{p+1}},
\end{align*} 
and so we are done.
\end{proof}
We are now all set to prove a generalization of Bressoud-Subbarao's identity, namely, Equation
\eqref{One_Var_BS-general}.

\subsection{Proof of Theorem \ref{BS-general one var}}
\begin{proof} Applying the differential operator $D$ successively $m$ times on both sides of identity \eqref{Part_Imp_Entry4}, we find that
\begin{align}
\sum_{\pi\in \mathcal{D}(n)}(-1)^{\#(\pi)-1}D^m\left(a^{\ell(\pi)-s(\pi)+1}\left(\frac{a^{s(\pi)}-1}{a-1}\right)\right) &=\sum_{d | n}D^m\left(a^d\right) \nonumber\\
\Rightarrow \sum_{\pi\in \mathcal{D}(n)}(-1)^{\#(\pi)-1}D^m \left( a^{\ell(\pi) - s(\pi)}\left(a + a^2 + \cdots + a^{s(\pi)}\right) \right)
&=\sum_{d|n}D^m\left(a^d\right).
\label{m_times_diff of Part_Imp_Entry4}
\end{align}
Note that $\ell(\pi) - s(\pi) \in \mathbb{N} \cup \{ 0 \}$ and $s(\pi) \in \mathbb{N}$ for 
each $\pi \in \mathcal{D}(n)$. So we are in a position to apply 
Lemma \ref{D^m lemma for BS-general}
to each of the polynomials $F_0(a; \pi):= a^{\ell(\pi) - s(\pi)}(a + a^2 + \cdots + a^{s(\pi)})$ 
as $\pi$ ranges over the collection $\mathcal{D}(n)$.
Also, for each $m \in \mathbb{N}$, we have  
\begin{align*}
D^m\left[a^d\right]=d^m a^d.
\end{align*} 
Thus the partition identity \eqref{m_times_diff of Part_Imp_Entry4} now becomes
\begin{align*}
\sum_{\pi\in \mathcal{D}(n)}(-1)^{\#(\pi)-1}D^m\left[F_0(a; \pi)\right]=\sum_{d  |  n}d^ma^d.
\end{align*}
Appealing to Lemma \ref{D^m lemma for BS-general}, we get
\begin{equation}\label{star}
\sum_{\pi\in \mathcal{D}(n)}(-1)^{\#(\pi)-1} \sum_{j=1}^{s(\pi)} (\ell(\pi)-s(\pi)+j)^m 
 a^{\ell(\pi)-s(\pi)+j}= \sum_{d  |  n}d^ma^d,
\end{equation}
which is nothing but identity \eqref{One_Var_BS-general} of Theorem \ref{BS-general one var} 
for $m \in \mathbb{N}$. Equation \eqref{Part_Imp_Entry4} is evidently the case $m=0$ of 
\eqref{One_Var_BS-general}. We are left with the case when $m$ is a negative integer. Suppose
$m=-u$, where $u$ is a positive integer. Here the 
operator $I$ plays an analogous role to that of $D$ in the above discussion. We now 
start by applying $I$ to \eqref{Part_Imp_Entry4}.
\begin{align}
\sum_{\pi\in \mathcal{D}(n)}(-1)^{\#(\pi)-1}I^u\left(a^{\ell(\pi)-s(\pi)+1}
\left(\frac{a^{s(\pi)}-1}{a-1}\right)\right) &=\sum_{d | n}I^u\left(a^d\right) \nonumber\\
\Rightarrow \sum_{\pi\in \mathcal{D}(n)}(-1)^{\#(\pi)-1}I^u \left( a^{\ell(\pi) - s(\pi)}
\left(a + a^2 + \cdots + a^{s(\pi)}\right) \right)
&=\sum_{d|n}I^u\left(a^d\right). \label{u_times_int of Part_Imp_Entry4}
\end{align}
We now apply Lemma \ref{I^m lemma for BS-general} to each of 
$F_0(a; \pi):= a^{\ell(\pi) - s(\pi)} \left(a + a^2 + \cdots + a^{s(\pi)}\right), \ 
\pi \in \mathcal{D}(n)$. Equation \eqref{u_times_int of Part_Imp_Entry4} now takes the form
\begin{align*}
\sum_{\pi\in \mathcal{D}(n)}(-1)^{\#(\pi)-1}I^u\left[F_0(a; \pi)\right]=
\sum_{\pi\in \mathcal{D}(n)}(-1)^{\#(\pi)-1}\sum_{j=1}^{s(\pi)} 
\frac{a^{\ell(\pi)-s(\pi)+j}}{(\ell(\pi)-s(\pi)+j)^u}=
\sum_{d|n}\frac{a^d}{d^u}.
\end{align*}
In other words, for negative integers $m$, we have
\begin{equation}\label{doublestar}
\sum_{\pi\in \mathcal{D}(n)}(-1)^{\#(\pi)-1}\sum_{j=1}^{s(\pi)} 
a^{\ell(\pi)-s(\pi)+j}(\ell(\pi)-s(\pi)+j)^m=
\sum_{d|n}a^d d^m.
\end{equation}
Combining \eqref{star} and \eqref{doublestar}, along with the case $m=0$, 
we complete the proof of Theorem \ref{BS-general one var}.
%
\end{proof}

\begin{remark}
It may seem artificial to differentiate the partition implication of a $q$-series identity rather than the identity itself (namely Entry $4$
in this case), which is the usual practice in the theory of $q$-series and partitions. But when we proceed to apply the operator $D$ 
successively $k$ times to Entry $4$, we arrive at
\begin{equation*}
\sum_{n=1}^{\infty}(-1)^{n-1} \frac{q^{\frac{n(n+1)}{2}}}{1-q^n}D^k 
\left( \frac{a^n}{(aq)_n} \right) = \sum_{n=1}^{\infty} \frac{n^k a^n q^n}{1-q^n}. 
\end{equation*}
In the above equation, we see that multiple differentiation on the left side
would become complicated owing to the presence of the $\frac{1}{(aq)_n}$ factor. 
This is mitigated by differentiating the partition implication instead, namely, Equation
\eqref{Part_Imp_Entry4}. This has an added advantage that, we could `reverse' the process and apply the integral operator $I$ too, thereby extending
Bressoud-Subbarao's identity to negative integers. To the best of our knowledge, this integral operator has not been applied to $q$-series identities directly.
\end{remark}

%

\section{Some Weighted Partition Identities}
On many occasions, it is seen that the inherent beauty of a $q$-series identity may not be 
completely realized until we throw some light on its
partition implications. This underlying motivation led us to investigate a few corollaries of our main results in search of more partition identities with some
nice weights. In this section, we derive some interesting partition-theoretic implications of Theorem \eqref{form 2 of the generalization}.
We let $d \rightarrow -1$ in Theorem \ref{form 2 of the generalization}, getting an identity which is closely associated to 
Dixit-Maji's identity \eqref{gen of Garvan},
\begin{equation}\label{WPI, d=-1}
\sum_{n=1}^{\infty} \frac{z^{n-1}(-cq)_{n-1} q^{\frac{n(n+1)}{2}}}{(zq)_n (cq)_n} =  \sum_{n=1}^{\infty} \frac{\left(\frac{-zq}{c}\right)_{n-1} c^{n-1}q^n}{(zq)_n}. 
\end{equation}
We now study the special cases of \eqref{WPI, d=-1}. One motivation is, in \cite{DM}, the authors considered the special cases
of their identity \cite[Equation (2.3)]{DM} and were fruitful in their endeavour. Since their identity is a special case of 
\eqref{form 2 of the generalization} with $d=1$, richer partition-theoretic information may still be waiting to be unlocked in the other cases of
\eqref{form 2 of the generalization}. 


\begin{proof}[Theorem \textup{\ref{two divisor functions}}][]
We begin by putting $c=z=-1$ in \eqref{WPI, d=-1}, which gives us
\begin{equation}\label{d=c=z=-1}
\sum_{n=1}^{\infty} \frac{(-1)^{n-1} (q)_{n-1} q^{\frac{n(n+1)}{2}}}{(-q)_n^2} = \sum_{n=1}^{\infty} \frac{(-1)^{n-1} q^{n}}{1+q^n}.
\end{equation}
Start with the right side above, writing it as 
\begin{equation}\label{rhsinter1}
 \sum_{n=1}^{\infty} (-1)^n \frac{-q^n}{1-(-q^n)}
= \sum_{n=1}^{\infty} (-1)^n \sum_{m=1}^{\infty} (-1)^m q^{m \cdot n} 
 =\sum_{k=1}^{\infty} \left( \sum_{d | k} (-1)^{d + \frac{k}{d}} \right) q^k.
\end{equation}
By referring to the sequence $A228441$ on the OEIS webpage \cite{sloane}, \cite[p.4, p.8]{glaisher}, we see that
$\sum_{d | k} (-1)^{d + \frac{k}{d}}$ is simply $d(k) - 4 d_{2, 4}(k)$. Coming to the left side of \eqref{d=c=z=-1},
\begin{align*}
\sum_{n=1}^{\infty} \frac{(-1)^{n-1} (q)_{n-1} q^{\frac{n(n+1)}{2}}}{(-q)_{n}^2} 
= \sum_{n=1}^{\infty} (-1)^{n-1} \frac{q(1-q)}{(1+q)^2} \frac{q^2(1-q^2)}{(1+q^2)^2} \cdots
\frac{q^{n-1}(1-q^{n-1})}{(1+q^{n-1})^2}\cdot\frac{q^n}{(1+ q^n)^2}.
\end{align*}
Now,
$\frac{q^r (1-q^r)}{(1+q^r)^2} = \sum_{k_r = 1}^{\infty} (-1)^{k_r - 1} (2k_r - 1) q^{k_r \cdot r}, \ (1 \leq r \leq n-1)$, 
$\frac{q^n}{(1+ q^n)^2} = \sum_{k_n = 1}^{\infty} (-1)^{k_n - 1} k_n q^{k_n \cdot n}.$
Multiplying all these terms together along with $(-1)^{n-1}$ we get, 
(here $n = \ell(\pi) = \nu_d(\pi)$, where $\pi \in \mathcal{P}^{\ast}(m)$)
\begin{align*}
 &\sum_{m=1}^{\infty} \left( \sum_{\pi \in \mathcal{P}^{\ast}(m)} (-1)^{\ell(\pi) - 1} (-1)^{\#(\pi) - \ell(\pi)} \nu(\ell(\pi))
 \prod_{i=1}^{\ell(\pi) - 1} (2\nu(i) - 1) \right) q^m \\
 &= \sum_{m=1}^{\infty} \left( \sum_{\pi \in \mathcal{P}^{\ast}(m)} (-1)^{\#(\pi) - 1}
\omega(\pi) \right) q^m,
\end{align*}
where $\omega(\pi)$ is as defined in \eqref{omega(pi)}.
By comparing coefficients of the above line with that of \eqref{rhsinter1}, we finally get, for $n \geq 1$,
\begin{equation*}
 \sum_{\pi \in \mathcal{P}^{\ast}(n)} (-1)^{\#(\pi) - 1} \omega(\pi) = d(n) - 4d_{2, 4}(n).
\end{equation*}
\end{proof}
\begin{proof}[Proposition \textup{\ref{WPI for overpartition}}][]
We let $c=z=1$ in \eqref{WPI, d=-1}, giving us the identity
\begin{equation}\label{d=-1, c=z=1}
\sum_{n=1}^{\infty} \frac{(-q)_{n-1} q^{\frac{n(n+1)}{2}}}{(q)_n^2} = \sum_{n=1}^{\infty} \frac{(-q)_{n-1} q^n}{(q)_n}.
\end{equation}
Firstly, see that the left side of \eqref{d=-1, c=z=1} can be written as
\begin{equation*}
\sum_{n=1}^{\infty} \frac{(1+q)q}{(1-q)^2} \cdots  \frac{(1+q^{n-1})q^{n-1}}{(1-q^{n-1})^2} \frac{q^n}{(1-q^n)^2}.
\end{equation*}
As before, note that $\frac{(1+q^r)q^r}{(1-q^r)^2} = \sum_{k_r = 1}^{\infty} (2k_r - 1) q^{k_r \cdot r}$ for $1 \leq r \leq n-1$.
Moreover, $\frac{q^n}{(1-q^n)^2} = \sum_{k_n = 1}^{\infty} = k_n q^{k_n \cdot n}$. Multiplying these series together and interpreting in terms of partitions,
the left side of \eqref{d=-1, c=z=1} thus becomes
\begin{equation}\label{combintLHS}
\sum_{n=1}^{\infty} \left( \sum_{\pi \in \mathcal{P}^{\ast}(n)} \nu(\ell(\pi)) \prod_{i=1}^{\ell(\pi)-1} (2\nu(i) -1) \right) q^n
= \sum_{n=1}^{\infty} \left( \sum_{\pi \in \mathcal{P}^{\ast}(n)} \omega(\pi) \right) q^n.
\end{equation}
Now, the right side of \eqref{d=-1, c=z=1} is
\begin{align*}
\sum_{n=1}^{\infty} \frac{(-q)_{n-1} q^n}{(q)_n} = \sum_{n=1}^{\infty} \frac{1+q}{1-q} \cdots \frac{1+q^{n-1}}{1-q^{n-1}}\frac{q^n}{1-q^n}. 
\end{align*}
Now, $\frac{1+q^r}{1-q^r} = 1 + \sum_{t=1}^{\infty} 2q^{tr} \quad (1 \leq r \leq n-1)$.
Multiplying together these series along with that corresponding to $n$, the interpretation
for the right side of \eqref{d=-1, c=z=1} is
\begin{equation}\label{combintRHS}
 \sum_{n=1}^{\infty} \left( \sum_{\pi \in \mathcal{P}(n)} 2^{\nu_d(\pi) - 1}\right) q^n.
\end{equation}
Comparing \eqref{combintLHS} and \eqref{combintRHS} gives us
\begin{equation*}
\sum_{\pi \in \mathcal{P}^{\ast}(n)} \omega(\pi) = \sum_{\pi \in \mathcal{P}(n)} 2^{\nu_d(\pi) - 1}.
\end{equation*}
This can also be written as
\begin{equation*}
2\sum_{\pi \in \mathcal{P}^{\ast}(n)} \omega(\pi) = \overline{p}(n),
\end{equation*}
by the well-known identity connecting the number of overpartitions $\overline{p}(n)$ and the weighted sum over partitions 
$\sum_{\pi \in \mathcal{P}(n)} 2^{\nu_d(\pi)}$, namely,  
$\sum_{\pi \in \mathcal{P}(n)} 2^{\nu_d(\pi)} = \overline{p}(n)$.
\end{proof}
\begin{proposition}
The following identity holds true for any $n \in \mathbb{N}$.
 \begin{equation*}
  \sum_{\pi \in \mathcal{P}(n)} (-1)^{\ell(\pi) - 1} 2^{\nu_d(\pi) - 1} =
  \sum_{\substack{\pi \in \mathcal{D}(n) \\ s(\pi) \ \text{odd}}} (-1)^{\#(\pi) - 1}. 
\end{equation*}
\end{proposition}
\begin{proof}
Making the substitutions $c \rightarrow 1, \ z \rightarrow -1$ in \eqref{WPI, d=-1} leads us to
\begin{equation}\label{d=-1, c=1, z=-1}
\sum_{n=1}^{\infty} \frac{(-1)^{n-1} q^{\frac{n(n+1)}{2}}}{(1+q^n) (q)_n} = \sum_{n=1}^{\infty}
\frac{(q)_{n-1} q^n}{(-q)_n}.
\end{equation}
The left side takes the form
\begin{align*}
\sum_{n=1}^{\infty} \frac{(-1)^{n-1} q^{1 + 2 + \cdots + n}}{(1-q) (1-q^2) \cdots (1-q^{n-1})(1-q^{2n})}
= \sum_{n=1}^{\infty} \left( \sum_{\substack{\pi \in \mathcal{P}^{\ast}(n) \\
\nu(\ell(\pi)) \ \text{odd}}} (-1)^{\ell(\pi) - 1} \right) q^n.
\end{align*}
On the other hand, the right side of \eqref{d=-1, c=1, z=-1} is
\begin{align*}
\sum_{n=1}^{\infty} \frac{(q)_{n-1} q^n}{(-q)_n} = 
\sum_{n=1}^{\infty} \frac{(1-q)(1-q^2) \cdots (1-q^{n-1}) q^n}{(1+q)(1+q^2) \cdots (1+q^{n-1})(1+q^n)}
= \sum_{n=1}^{\infty} \left( \sum_{\pi \in \mathcal{P}(n)} (-1)^{\#(\pi) - 1} 
 2^{\nu_d(\pi) - 1} \right) q^n.
\end{align*}
By comparing coefficients, we get
\begin{equation*}
 \sum_{\pi \in \mathcal{P}(n)} (-1)^{\#(\pi) - 1} 
 2^{\nu_d(\pi) - 1} = \sum_{\substack{\pi \in \mathcal{P}^{\ast}(n) \\ \nu(\ell(\pi)) \ \text{odd} }}
(-1)^{\ell(\pi) - 1},
 \end{equation*}
which upon conjugation yields
\begin{equation*}
  \sum_{\pi \in \mathcal{P}(n)} (-1)^{\ell(\pi) - 1} 2^{\nu_d(\pi) - 1} =
  \sum_{\substack{\pi \in \mathcal{D}(n) \\ s(\pi) \ \text{odd}}} (-1)^{\#(\pi) - 1}. 
\end{equation*}
This is the same as Equation $(5.8)$ in \cite{DM}.
\end{proof}
\subsection{Unexplored identities by letting $c \rightarrow 0$ in Theorem $2.1$ of \cite{DM}}
While looking for partition-theoretic consequences that arise from our main result,
Equation \ref{CEM-main identity}, we stumbled upon an interesting one-parameter identity which
arises out of the special case $d=1$, namely from \cite[Theorem 2.1]{DM}. In the present subsection, we derive this result.
\begin{proposition}\label{crightarrow0DM}
Let $\mathcal{P}_1^{\ast}(n)$ denote the collection of partitions of $n$ in which each integer lying between the smallest and largest parts of a 
partition also occurs as a part. Then,
\begin{equation}\label{c=0DMfinalid}
 \sum_{\substack{\pi \in \mathcal{P}(n) \\ s(\pi) \geq 2}}
 t^{\ell(\pi) + \#(\pi) - 2}\left( 1-\frac{1}{t}\right)^{\nu_d(\pi) - 1} =
 \sum_{\substack{\pi \in \mathcal{P}(n) \\ \ell(\pi) = 2}} t^{\#(\pi)} + 
 \sum_{\substack{\pi \in \mathcal{P}_1^{\ast}(n) \\ s(\pi)= 2, \ \ell(\pi) \geq 3}} 
 (-1)^{\ell(\pi)} t^{\#(\pi)} \nu(\ell(\pi) - 1).
\end{equation} 
\end{proposition}
\begin{proof}
Firstly, we let $d \rightarrow 1, \ c \rightarrow 0$ in \eqref{CEM-main identity}. This gives us
\begin{equation}\label{c=0DM}
 \sum_{n=1}^{\infty} \frac{\left(\frac{b}{a}\right)_n a^n}{(b)_n} = 
 \sum_{m=0}^{\infty} \frac{(-1)^m q^{\frac{m(m-1)}{2}} b^m}{(b)_m} \left( \frac{aq^m}{1-aq^m}
 - \frac{bq^m}{1-bq^m}\right).
\end{equation}
We first consider the right side of \eqref{c=0DM}. Put $a=tq, b=tq^2$ to get
\begin{align*}
 \sum_{m=0}^{\infty} \frac{(-1)^m t^m q^{\frac{m(m+1)}{2}+m}}{(tq^2)_m}
 \left( \frac{tq^{m+1}}{1-tq^{m+1}}
 - \frac{tq^{m+2}}{1-tq^{m+2}}\right).
\end{align*}
Separating the terms corresponding to $m=0$ and $m=1$, we get
\begin{align*}
& \left( \frac{tq}{1-tq} - \frac{tq^2}{1-tq^2}\right) + \left( \frac{-tq^2}{1-tq^2}
\left\{ \frac{tq^2}{1-tq^2} - \frac{tq^3}{1-tq^3} \right\} \right) 
+ \sum_{m=2}^{\infty} \frac{(-1)^m t^m q^{\frac{m(m+1)}{2}+m} tq^{m+1} (1-q)}
{(tq^2)_m (1-tq^{m+1})(1-tq^{m+2})}\\
&= \frac{tq(1-q)}{(1-tq)(1-tq^2)} - \frac{(tq^2)^2 (1-q)}{(1-tq^2)^2 (1-tq^3)} + 
\frac{1-q}{q} \sum_{m=2}^{\infty} \frac{(-1)^m t^{m+1} q^{2(m+1)} q^{\frac{m(m+1)}{2}}}{(tq^2)_m (1-tq^{m+1})(1-tq^{m+2})}.
\end{align*}
Now, multiplying by $\frac{q}{1-q}$, we have
\begin{equation}\label{3terms}
\frac{tq^2}{(1-tq)(1-tq^2)} - \frac{(tq^2)(tq^3)}{(1-tq^2)^2(1-tq^3)} +
\sum_{m=2}^{\infty}  \frac{(-1)^m t^{m+1} q^{2(m+1)} q^{\frac{m(m+1)}{2}}}{(tq^2)_m (1-tq^{m+1})(1-tq^{m+2})}. 
\end{equation}
In \eqref{3terms}, the first and second terms can be interpreted respectively as 
\begin{equation}\label{1stand2ndterms}
 \sum_{n=1}^{\infty} \left( \sum_{\substack{\pi \in \mathcal{P}(n) \\ \ell(\pi) = 2}} t^{\#(\pi)}\right) q^n \quad
\text{and} \quad \sum_{n=1}^{\infty} \left( \sum_{\substack{\pi \in \mathcal{P}(n) \\ s(\pi) = 2,\ \ell(\pi) = 3}} t^{\#(\pi)} \nu(2)\right) q^n.  
\end{equation}
Let us now focus on the third term in \eqref{3terms}, namely,
\begin{equation}\label{3rdterm}
\sum_{m=2}^{\infty} \frac{(-1)^m t^{m+1} q^{2(m+1)} q^{\frac{m(m+1)}{2}}}{(tq^2)_m (1-tq^{m+1})(1-tq^{m+2})}
= \sum_{m=2}^{\infty} \frac{(-1)^m \left(tq^2\right)\left(tq^3\right) \cdots 
\left(tq^{m+2}\right)}{(1-tq^2) \cdots (1-tq^m)(1-tq^{m+1})^2(1-tq^{m+2})}.
\end{equation}
Now, for $2 \leq i \leq m+2, \ i \neq m+1$, we have 
$\frac{tq^i}{1-tq^i} = \sum_{k_i = 1}^{\infty} t^{k_i} q^{k_i \cdot i},\
\frac{tq^{m+1}}{(1-tq^{m+1})^2} =   \sum_{k_{m+1} = 1}^{\infty}
k_{m+1} t^{k_{m+1}} q^{k_{m+1} \cdot (m+1)}.$ 
Combining all these terms, \eqref{3rdterm} becomes
\begin{equation}\label{3rdtermint}
 \sum_{n=1}^{\infty} \left( \sum_{\substack{\pi \in 
 \mathcal{P}_1^{\ast}(n) \\ s(\pi) = 2, \ \ell(\pi) \geq 4}} 
 (-1)^{\ell(\pi)} t^{\#(\pi)} \nu(\ell(\pi)-1) \right) q^n.
\end{equation}
Thus, from \eqref{1stand2ndterms} and \eqref{3rdtermint},
the complete interpretation of \eqref{3terms} is as follows:
\begin{align}
&\sum_{\substack{\pi \in 
 \mathcal{P}(n) \\ \ell(\pi) = 2}} t^{\#(\pi)} +
\sum_{\substack{\pi \in 
 \mathcal{P}(n) \\ s(\pi) = 2, \ \ell(\pi) = 3}} 
 (-1)^{\ell(\pi)} t^{\#(\pi)} \nu(\ell(\pi)-1) +
 \sum_{\substack{\pi \in 
 \mathcal{P}_1^{\ast}(n) \\ s(\pi) = 2, \ \ell(\pi) \geq 4}} 
 (-1)^{\ell(\pi)} t^{\#(\pi)} \nu(\ell(\pi)-1) \nonumber \\
& =\sum_{\substack{\pi \in 
 \mathcal{P}(n) \\ \ell(\pi) = 2}} t^{\#(\pi)} +
 \sum_{\substack{\pi \in 
 \mathcal{P}_1^{\ast}(n) \\ s(\pi) = 2, \ \ell(\pi) \geq 3}} 
 (-1)^{\ell(\pi)} t^{\#(\pi)} \nu(\ell(\pi)-1). \label{PI_RHS}
 \end{align}
This is nothing but the partition-theoretic meaning extracted from the right
hand side of \eqref{c=0DM} after making the substitutions $a=tq, b=tq^2$ and then
multiplying throughout by $\frac{q}{1-q}$.
We now try to interpret the left hand side of \eqref{c=0DM} after subjecting it through
the same sequence of operations. This brings us to
\begin{align*}
\frac{q}{1-q} \sum_{n=1}^{\infty} \frac{(q)_n t^n q^n}{(tq^2)_n}
=\sum_{n=1}^{\infty} \frac{\{t(1-q^2)\} \cdots \{ t(1-q^n)\} (tq^{n+1})}{(1-tq^2)\cdots(1-tq^n)(1-tq^{n+1})}.
\end{align*}
Now, for $2 \leq r \leq n$, we have
$\frac{t(1-q^r)}{1-tq^r} = t + \sum_{k_r=1}^{\infty} t^{k_r}(t-1) q^{k_r \cdot r}
= t\left( 1 + \sum_{k_r = 1}^{\infty} t^{k_r} \left( 1-\frac{1}{t}\right)q^{k_r \cdot r}\right)$.
Again, $ \frac{tq^{n+1}}{1-tq^{n+1}} = \sum_{k_{n+1} = 0}^{\infty} t^{k_{n+1}}
q^{k_{n+1}\cdot (n+1)}$. 
Multiplying all these together we obtain
\begin{equation}\label{PI_LHS}
 \sum_{n=2}^{\infty} \left( \sum_{\substack{\pi \in \mathcal{P}(n) \\ s(\pi) \geq 2}}
 t^{\ell(\pi)-2} \cdot t^{\#(\pi)} \cdot \left( 1-\frac{1}{t}\right)^{\nu_d(\pi) - 1} \right) q^n. 
\end{equation}
Comparing \eqref{PI_LHS} and \eqref{PI_RHS}, we obtain \eqref{c=0DMfinalid}.
\end{proof}
We now consider a special case of Proposition \ref{crightarrow0DM}, namely, letting $t \rightarrow 1$ in \eqref{c=0DMfinalid}.
\begin{proof}[Theorem \textup{\ref{analogous to FFW}}][]
As we tend $t \rightarrow 1$ in the left side of \eqref{c=0DMfinalid}, only the terms corresponding to $\nu_d(\pi)=1$ will survive. So,
\begin{equation*}
\lim_{t \rightarrow 1} \sum_{\substack{\pi \in \mathcal{P}(n) \\ s(\pi) \geq 2}}t^{\ell(\pi) + \#(\pi) - 2}\left( 1-\frac{1}{t}\right)^{\nu_d(\pi) - 1} 
=\sum_{\substack{\pi \in \mathcal{P}(n) \\ s(\pi) \geq 2, \ \nu_d(\pi) = 1}} 1.
\end{equation*}
But a partition $\pi$ of $n$ with $\nu_d(\pi) = 1$ simply corresponds to a divisor of $n$. For example, if 
$n=6$, the partitions $\pi$ with $\nu_d(\pi)=1$ are $6, \ 3+3, \ 2+2+2, 1+1+1+1+1+1$, corresponding to the divisors
$6, 3, 2, 1$ of $6$. Since we have the additional condition that $s(\pi) \geq 2$, we have to omit the 
partition with all $1$'s, or equivalently, we count the number of divisors of $n$ which are greater than $1$, which is simply
$d(n) - 1$. Thus, the left side of \eqref{c=0DMfinalid} reduces to $d(n) - 1$ in the case when $t \rightarrow 1$.
We now tend $t \rightarrow 1$ in the right side of \eqref{c=0DMfinalid}, 
\begin{equation}\label{rhSint}
\sum_{\substack{\pi \in \mathcal{P}(n) \\ \ell(\pi) = 2}} 1 + 
\sum_{\substack{\pi \in \mathcal{P}_1^{\ast}(n) \\ s(\pi)= 2, \ \ell(\pi) \geq 3}} 
 (-1)^{\ell(\pi)} \nu(\ell(\pi) - 1).
\end{equation}
The first sum here is simply the number of partitions of $n$ with $\ell(\pi)=2$. So
these partitions contain all $2$'s  or a combination of $2$'s and $1$'s. For example,
consider $n=7$. The relevant partitions are $2+2+2+1, \ 2+2+1+1+1, \ 2+1+1+1+1+1$. Once the number
of $2$'s in the partition is known, the partition is completely determined 
(since the remaining parts, if any, are all $1$'s). By keeping track of 
the number of $2$'s in the partition we see that partitions with largest part $2$ are either
$n/2$ in number (if $n$ is even) or $(n-1)/2$ in number (if $n$ is odd). In either case, this reduces to
$\lfloor\frac{n}{2}\rfloor$.

Coming to the second sum in \eqref{rhSint}, we conjugate the set of partitions over which the sum runs,
\begin{align*}
 \sum_{\substack{\pi \in \mathcal{P}_1^{\ast}(n) \\ s(\pi)= 2, \ \ell(\pi) \geq 3}} 
 (-1)^{\ell(\pi)} \nu(\ell(\pi) - 1) = \sum_{\substack{\pi \in \mathcal{D}_1(n) \\ 
 \nu(\ell(\pi)) = 2, \ \#(\pi) \geq 3}} (-1)^{\#(\pi)} (s_2(\pi) - s(\pi)),
\end{align*}
where $\mathcal{D}_1(n)$ and $s_2(\pi)$ are as defined in Theorem \ref{analogous to FFW}.
Putting our observations together, the case $t \rightarrow 1$ of the identity \eqref{c=0DMfinalid} thus gives us
\begin{equation*}
d(n) = 1 + \left\lfloor\frac{n}{2}\right\rfloor - \sum_{\pi \in \mathcal{D}^{\ast}(n)} (-1)^{\#(\pi) - 1}
(s_2(\pi) - s(\pi)).
\end{equation*}
\end{proof}
\section{Concluding Remarks}
In this article, we have established a one-variable generalization of the main identity of Dixit and Maji \cite[Theorem 2.1]{DM}. From this, 
all five entries of Ramanujan now follow as corollaries. We have seen that from one of these entries, namely Entry $4$, we were able to derive a 
one-parameter generalization of the identity of Bressoud and Subbarao, namely, Theorem \ref{BS-general one var}, 
by using the operators $D$ and $I$. From the generalized identity, that is, Theorem \ref{Generalization of Dixit-Maji},
we found a few interesting partition-theoretic implications, including some in the spirit of 
Bressoud and Subbarao, connecting certain divisor functions with weighted partition functions.

It would be desirable to find bijective proofs of these identities, for example, say \eqref{c=0DM, t=1} and \eqref{BSanalogue_a=-1}, which is 
the identity analogous to that of Bressoud and Subbarao obtained by putting $a=-1$ in \eqref{One_Var_BS-general}. Readers can further investigate the operators $D$ and $I$, 
especially in connection to newer weighted partition identities of Bressoud-Subbarao type. 
It would also be fascinating to see if the one-parameter generalization of Bressoud-Subbarao's identity, namely, \eqref{One_Var_BS-general} 
holds for any complex number $m$, and not just for integers. Offcourse, in that case, we have to take recourse to some other technique.

\end{document}